\documentclass[a4paper,11pt]{article}

\usepackage{amsmath, amsthm, amssymb, color, enumitem, graphicx}

\usepackage[nobysame, initials]{amsrefs}

\usepackage[draft=false]{hyperref}
\hypersetup{
    colorlinks=true,        
    linkcolor=black,        
    citecolor=black,        
    filecolor=black,        
    urlcolor=black          
}

\usepackage[a4paper]{geometry}

\usepackage{bm}

\def\itm#1{\rm ({#1})} 
\def\itmit#1{\itm{\it #1\,}} 
\def\rom{\itmit{\roman{*}}}

\def\endofClaim{\hfill\scalebox{.6}{$\Box$}}

\def\vi{\bm{i}}
\def\l{\ell}

\def\phi{\varphi}
\def\le{\leqslant}
\def\ge{\geqslant}
\def\CC{\mathbb{C}}

\def\FF{\mathbb{F}}

\def\PP{\mathbb{P}}

\def\RR{\mathbb{R}}

\def\ZZ{\mathbb{Z}}
\newcommand{\setbuilder}[2]{\left\{#1:#2\right\}}

\DeclareMathOperator{\rank}{rank}
\DeclareMathOperator{\cat}{cat}

\newtheorem{theorem}{Theorem}[section]
\newtheorem{lemma}[theorem]{Lemma}
\newtheorem{corollary}[theorem]{Corollary}
\newtheorem{prop}[theorem]{Proposition}

\theoremstyle{definition}

\theoremstyle{remark}

\newcommand{\oldqed}{}

\title{Ordinary planes, coplanar quadruples, and space quartics}
\author{Aaron Lin\footnote{Department of Mathematics, London School of Economics and Political Science, United Kingdom.} \and Konrad Swanepoel\footnotemark[1]}
\date{}

\begin{document}

\maketitle

\begin{abstract}
An ordinary plane of a finite set of points in real $3$-space with no three collinear is a plane intersecting the set in exactly three points.
We prove a structure theorem for sets of points spanning few ordinary planes.
Our proof relies on Green and Tao's work on ordinary lines in the plane, combined with classical results on space quartic curves and non-generic projections of curves.
This gives an alternative approach to Ball's recent results on ordinary planes, as well as extending them.
We also give bounds on the number of coplanar quadruples determined by a finite set of points on a rational space quartic curve in complex $3$-space, answering a question of Raz, Sharir and De Zeeuw [Israel J. Math.\ \textbf{227} (2018)].
\end{abstract}

\section{Introduction}\label{sec:intro}
Let $P$ denote a finite set of points in $3$-dimensional Euclidean (or real projective) space.
An \emph{ordinary line} of $P$ is a line passing through exactly two points of $P$.
The classical Sylvester-Gallai theorem states that any finite non-collinear point set in the plane has an ordinary line.
Green and Tao~\cite{GT13} proved a structure theorem for sets with few ordinary lines, and used it to prove the so-called Dirac-Motzkin conjecture: 
any sufficiently large non-collinear $n$-point set in the plane has at least $n/2$ ordinary lines.

It is natural to ask the same question of ordinary planes, where an \emph{ordinary plane} of $P$ is a plane passing through exactly three points of $P$. 
However, as pointed out by Motzkin~\cite{M51}, there are finite sets of points spanning $3$-space that do not have any ordinary planes.
His one example consists of the ten intersection points of triples of five planes in general position, and another consists of points chosen from two skew lines.
Motzkin~\cite{M51} suggested we consider instead planes $\Pi$ with all but one of the points of $P \cap \Pi$ lying on a line.
He then recovers an analogue of the Sylvester-Gallai theorem:
any finite non-coplanar set in $3$-space spans such a plane.

Purdy and Smith \cite{PS10} considered instead non-coplanar point sets in $3$-space that are in general position in the sense that no three points are collinear, proving a quadratic lower bound on the number of ordinary planes of the set.
Recently, Ball~\cite{B17} proved a structure theorem for sets with few ordinary planes by examining the so-called tetra-grid structure of the projective dual of such sets,
which is a $3$-dimensional analogue of the triangular grid used by Green and Tao~\cite{GT13}.
For a $4$-dimensional generalisation of Ball's ideas, see the very recent paper by Ball and Jimenez \cite{BJ18}.

The first aim of this paper is to give an alternative proof to a slightly more refined version of Ball's structure theorem in $3$-space (but with a stronger condition).
We do this by using results on non-generic projections of arbitrary space curves (see Section~\ref{sec:tools}), and by detailed considerations of space quartic curves, including notions from classical invariant theory (see Section~\ref{sec:quartics}).
We avoid the use of Ball's tetra-grids.
Our main result is Theorem~\ref{thm:strong} below.
A \emph{prism} is any set projectively equivalent to the vertex set of a prism over a regular polygon, which is a subset of $\RR^3$ of $n=2m$ points of the form
\[ \setbuilder{\left( \cos \left( \frac{2k\pi}{m} \right), \sin \left( \frac{2k\pi}{m} \right), \pm 1 \right)}{k=0, \dotsc, m-1}. \]
An \emph{antiprism} is any set projectively equivalent to the vertex set of an antiprism over a regular polygon, that is, a subset of $\RR^3$ of $n=2m$ points of the form
\begin{align*}
&\quad \setbuilder{\left( \cos \left( \frac{2k\pi}{m} \right), \sin \left( \frac{2k\pi}{m} \right), 1 \right)}{k=0, \dotsc, m-1}\\
&\cup \setbuilder{\left( \cos \left( \frac{2(k+1)\pi}{m} \right), \sin \left( \frac{2(k+1)\pi}{m} \right), -1 \right)}{k=0, \dotsc, m-1}.
\end{align*}
See also~\cite{B17}*{Section~2}.
In this theorem we refer to certain space quartic curves that have a natural group structure as described in Section~\ref{sec:quartics}. 
We prove this theorem in Section~\ref{sec:proof}.

\begin{theorem}[Full structure theorem]\label{thm:strong}
Let $K>0$ and suppose $n \ge C\max\{K^8,1\}$ for some sufficiently large constant $C>0$.
Let $P$ be a set of $n$ points in real projective $3$-space with no three collinear. If $P$ spans at most $Kn^2$ ordinary planes,
then up to a projective transformation, $P$ differs in at most $O(K)$ points from a configuration of one of the following types:
\begin{enumerate}[label=\rom]
\item A subset of a plane;
\item A prism or an antiprism;
\item A coset of a subgroup of an elliptic or acnodal space quartic curve.
\end{enumerate}
\end{theorem}

Note that it is simple to show that conversely, any of the three types of sets described in Theorem~\ref{thm:strong} span $O(Kn^2)$ ordinary planes.

Theorem~\ref{thm:strong} forms the basis for proving its higher-dimensional analogue, which we do in a subsequent paper \cite{LS}.

By stereographic projection we immediately obtain the following strengthening of~\cite{LMMSSZ17}*{Theorem~1.5}.
An \emph{ordinary circle} of a finite set of points in the plane is a circle passing through exactly three points of the set.
For definitions of the group on an ellipse, of a circular cubic curve, and of a double polygon, see \cite{LMMSSZ17}.
In \cite{LMMSSZ17}*{Theorem~1.5}, we need $n \ge \exp\exp(CK^C)$; here we only assume $n \ge CK^8$.

\begin{corollary}
Let $K>0$ and suppose $n \ge C\max\{K^8,1\}$ for some sufficiently large constant $C>0$.
Let $P$ be a set of $n$ points in the Euclidean plane.
If $P$ spans at most $Kn^2$ ordinary circles,
then up to inversions and similarities of the plane, $P$ differs in at most $O(K)$ points from a configuration of one of the following types:
\begin{enumerate}[label=\rom]
\item A subset of a line;
\item A coset of a subgroup of an ellipse;
\item A coset of a subgroup of a smooth circular cubic curve;
\item A double polygon that is `aligned' or `offset'.
\end{enumerate}
\end{corollary}

Theorem~\ref{thm:strong} also gives us the means to determine the maximum number of \emph{$4$-point planes}, which are planes passing through exactly four points of a given point set. 
The following corollary is an analogue of the orchard problem for circles \cite{LMMSSZ17}*{Theorem~1.3}.
It also immediately implies \cite{LMMSSZ17}*{Theorem~1.3} via stereographic projection.
We omit the proof, which goes along the same lines as the proof of \cite{LMMSSZ17}*{Theorem~1.3}.
(Note that the exact number in the case $n\equiv 0\pmod{8}$ was calculated incorrectly, but corrected in the arXiv version of \cite{LMMSSZ17}.)

\begin{corollary}
\mbox{}
\begin{enumerate}[label=\rom]
\item If $n$ is sufficiently large, the maximum number of $4$-point planes determined by a set of $n$ points in real projective $3$-space with no three collinear is equal to
\begin{equation*}
\begin{cases}
\frac{1}{24}n^3 - \frac14 n^2 + \frac{5}{6}n   & \text{if } n \equiv 0 \pmod{8},\\
\frac{1}{24}n^3 - \frac14 n^2 + \frac{11}{24}n - \frac{1}{4}  & \text{if } n \equiv 1, 3, 5, 7 \pmod{8},\\
\frac{1}{24}n^3 - \frac14 n^2 + \frac{7}{12}n - \frac12	& \text{if } n \equiv 2, 6 \pmod{8},\\
\frac{1}{24}n^3 - \frac14 n^2 + \frac56n - 1 & \text{if } n \equiv 4 \pmod{8}.
\end{cases}
\end{equation*}
\item Let $C$ be a sufficiently large constant.
If a set $P$ of $n$ points in real projective $3$-space with no three collinear determines more than $n^3/24 - 7n^2/24 + Cn$ $4$-point planes, then up to a projective transformation, $P$ lies on an elliptic or acnodal space quartic curve.
\end{enumerate}
\end{corollary}

The second aim of this paper is to consider the number of \emph{coplanar quadruples} (four distinct coplanar points) of an $n$-point set on quartic curves in complex $3$-space.
Raz, Sharir, and De Zeeuw~\cite{RSZ16} showed that such a set spans $O(n^{8/3})$ coplanar quadruples unless the curve contains a planar or a quartic component. 
They left it as an open problem whether there always exist configurations on rational space quartic curves spanning $\Theta(n^3)$ coplanar quadruples.
As mentioned in \cite{RSZ16}, certain space quartic curves (elliptic normal curves) contain sets of $n$ points with $\Theta(n^3)$ coplanar quadruples.
The properties of space quartic curves that we need to prove Theorem~\ref{thm:strong} also enable us to solve this problem.
We prove the following theorem in Section~\ref{sec:quad}.
Any space quartic curve is contained in a quadric surface (see Section~\ref{sec:quartics}).
If the quadric is unique, the curve is called a \emph{space quartic of the second species}, otherwise it is a \emph{space quartic of the first species}.
Elliptic normal curves are always of the first species, while space quartics of the second species are always rational (see Section~\ref{sec:quartics}).

\begin{theorem}\label{thm:quad}
Let $\delta$ be a rational space quartic curve in $\CC\PP^3$.
If $\delta$ is of the first species, then there exist $n$ points on $\delta$ that span $\Theta(n^3)$ coplanar quadruples.
If $\delta$ is of the second species, then any $n$ points on $\delta$ span $O(n^{8/3})$ coplanar quadruples.
\end{theorem}

\subsection*{A note on our methods}

Although all our results can be formulated for affine space, it is both more natural and more convenient to work in projective space, and this is what we will do.
Also, although we will mostly work with real varieties (mostly curves), we will always consider them to be the real points of a variety in complex projective space.

In classical algebraic geometry, many results are formulated for smooth varieties and for generic points.
For example, the projection of a curve of degree $d$ in  complex projective $3$-space from a generic point not on the curve onto a plane is again a curve of degree $d$.
Since we will project from given points not of our choosing, we cannot always assume that a projection point is generic, and will have to use more subtle results on projections (see Section~\ref{ssec:projections}).
Also, since we are working with an extremal problem, there is no guarantee that the curves on which we will find the points to lie are smooth; on the contrary, we should not be surprised that singularities occur in extremal objects.
Although it turns out that the curves that we consider will in the generic case be smooth quartics, quartic curves with singularities will also appear.
Thus we will have to use detailed classical results on rational space quartics.

\subsection*{Notation}

Let $\FF$ denote the field of real or complex numbers, and let $\FF\PP^m$ denote the $m$-dimensional projective space over $\FF$.
We only consider dimensions $m=2, 3$.
We denote the homogeneous coordinates of a point in $3$-dimensional projective space by a $4$-dimensional vector $[x_0, x_1, x_2, x_3]$.
We denote the algebraic surface where a homogeneous polynomial $f\in\FF[x_0, x_1, x_2, x_3]$ vanish by $Z_{\FF}(f)$.
More generally, we consider a \emph{variety} to be any intersection of algebraic surfaces.
We say that a variety is pure-dimensional if each of its irreducible components has the same dimension.
We consider a \emph{curve} of degree $d$ in $\CC\PP^3$ to be a variety $\delta$ of pure dimension $1$ such that a generic hyperplane of $\CC\PP^3$ intersects $\delta$ in $d$ distinct points.
We denote the Zariski closure of a set $S \subseteq \CC\PP^3$ by $\overline{S}$.

\section{Tools from classical algebraic geometry}\label{sec:tools}

\subsection{B\'ezout's theorem}\label{ssec:bezout}

We will use 
a form of B\'ezout's theorem \cite{H92}*{Theorem~18.4}.
Two pure-dimensional varieties $X$ and $Y$ in $\CC\PP^m$
are said to \emph{intersect properly} if $\dim(X\cap Y) = \dim(X) + \dim(Y) - m$.

\begin{theorem}[B\'ezout's theorem]\label{thm:Bezout}
Let $X,Y\subset\CC\PP^m$ 
be pure-dimensional varieties that intersect properly.
If the intersection $X\cap Y$ has irreducible components $\delta_1, \delta_2, \dotsc, \delta_k$, then their total degree $\sum_{i=1}^k\deg(\delta_i)$ is at most $\deg(X)\deg(Y)$.
\end{theorem}

By counting the multiplicity of each irreducible component $\delta_i$, it is possible to formulate a version of this theorem where equality holds \cite{H92}*{Theorem~18.4}, but we do not need this strengthening.
We will instead use the following corollaries of Theorem~\ref{thm:Bezout}.

\begin{corollary}\label{cor:Bezout1}
Let $X$ and $Y$ be two distinct surfaces in $\CC\PP^3$.
Then the irreducible components $\delta_1, \delta_2, \dotsc, \delta_k$ of $X\cap Y$ are all curves of total degree $\sum_{i=1}^k\deg(\delta_i)$ at most $\deg(X)\deg(Y)$.
\end{corollary}

\begin{proof}
Since the surfaces are distinct, their intersection will have lower dimension.
It then follows by the projective dimension theorem \cite{H77}*{Theorem~7.2} that $\dim(X\cap Y)=1$, so $X$ and $Y$ intersect properly.
The corollary now follows from Theorem~\ref{thm:Bezout}.
\end{proof}

The following corollary is an immediate consequence of the previous one.
\begin{corollary}\label{cor:Bezout2}
Let $\delta$ be a curve of degree $d_1$
and $X$ a surface of degree $d_2$ 
in $\CC\PP^3$ not containing any irreducible component of $\delta$.
Then $\delta \cap X$ has at most $d_1d_2$ points.
\end{corollary}

We present a proof of the following well-known result after we have discussed projections in Section~\ref{ssec:projections}.
\begin{corollary}\label{cor:Bezout3}
Two distinct curves $\delta_1$ and $\delta_2$ in $\CC\PP^3$ with no common irreducible component and of degrees $d_1$ and $d_2$, respectively, intersect in at most $d_1d_2$ points.
\end{corollary}

\subsection{Projections of curves}\label{ssec:projections}

We consider the projection from a point $p \in \FF\PP^3$ to be the mapping $\pi_p$ from $\FF\PP^3 \setminus \{p\}$ onto a plane $\Pi$ of $\FF\PP^3$ not containing $p$, where a point $q$ is mapped to the intersection of the line $pq$ with~$\Pi$ \cite{H92}*{Example~3.4}.
The exact choice of $\Pi$ does not matter, although certain choices might be more convenient.

Let $P$ be a finite point set in $\RR\PP^3$ with no three points collinear.
If we project $P \setminus \{p\}$ from a point $p \in P$, all ordinary planes of $P$ through $p$ map to ordinary lines of $\pi_p(P \setminus \{p\})$ in $\RR\PP^2$.
Green and Tao showed that sets with few ordinary lines lie on (possibly reducible) planar cubic curves \cite{GT13}*{Theorem 1.5}.
It turns out that we need to understand projections of curves so that we can apply their result to the structure of the set $P$.

Let $\delta$ be an irreducible non-planar curve of degree $d$ and $p$ a point in $\CC\PP^3$.
We say that $\pi_p$ is \emph{generically one-to-one on $\delta$} if there is a finite subset $S$ of $\delta$ such that $\pi_p$ restricted to $\delta\setminus S$ is one-to-one.
(This is equivalent to the birationality of $\pi_p$ restricted to $\delta\setminus\{p\}$ \cite{H92}*{p.~77}.)
It is a well-known classical fact that, if $\pi_p$ generically one-to-one, the degree of the curve $\overline{\pi_p(\delta \setminus \{p\})}$ is $d-1$ if $p$ lies on $\delta$, and is $d$ if $p$ does not lie on~$\delta$~\cite{H92}*{Example 18.16}, \cite{Kollar}*{Section 1.15}.
We can now prove Corollary~\ref{cor:Bezout3}, which is essentially a consequence of B\'ezout's theorem (Theorem~\ref{thm:Bezout}).

\begin{proof}[Proof of Corollary~\ref{cor:Bezout3}]
Since $\delta_1$ and $\delta_2$ are distinct with no common irreducible component, their intersection is finite.
Projecting from a generic point $p \in \CC\PP^3\setminus{(\delta_1 \cup \delta_2)}$, we thus have two planar curves $\delta_1' := \pi_p(\delta_1)$ and $\delta_2' := \pi_p(\delta_2)$ of degrees $d_1$ and $d_2$, respectively, which intersect properly.
The corollary then follows from B\'ezout's theorem (Theorem~\ref{thm:Bezout}).
\end{proof}

In the projections that we will make, we will not have complete freedom in choosing a projection point, and therefore we cannot guarantee that $\pi_p$ is generically one-to-one on $\delta$.
For this reason, we will need more sophisticated results on the projection of curves from a point.
We start with the following more elementary proposition, which is a restatement of \cite{B17}*{Lemma~6.6}. 

\begin{prop}\label{prop:cones}
Let $\sigma_1$ and $\sigma_2$ be two irreducible conics given by the intersection of two distinct planes and a quadric surface in $\CC\PP^3$.
Then there are at most two quadric cones containing both $\sigma_1$ and $\sigma_2$.
\end{prop}

We define a \emph{trisecant} of an irreducible non-planar curve $\delta$ in $\CC\PP^3$ to be a line that intersects $\delta$ in at least three distinct points, or that can be approximated in the Zariski topology by such lines.
More precisely, note that lines in $\CC\PP^3$ can be parametrised using Pl\"ucker coordinates by points on the Klein quadric in $\CC\PP^5$ \cite{H92}*{Example~6.3}.
A \emph{trisecant} of $\delta$ is a line that corresponds to a point in the Zariski closure of the set of points on the Klein quadric that correspond to lines that intersect $\delta$ in at least three distinct points \cite{H92}*{Example~8.9}.
The classical trisecant lemma states that the number of points on an irreducible non-planar curve in $\CC\PP^3$ that lie on infinitely many trisecants is finite \cite{ACGH}*{pp.~109--111}, \cite{Severi}*{p.~85, footnote}.
We first state a generalisation (Lemma~\ref{lem:trisecant}) of the trisecant lemma to curves that are not necessarily irreducible, which is a simple consequence of \cite{KKT08}*{Theorem 2}. 
We then prove three quantitative versions of the trisecant lemma (Lemmas \ref{lem:projection2}, \ref{lem:projection3}, and \ref{lem:projection1}).

\begin{lemma}[Trisecant lemma]\label{lem:trisecant}
Let $\delta$ be a non-planar curve in $\CC\PP^3$.
Then the number of points on $\delta$ that lie on infinitely many trisecants of $\delta$ is finite.
\end{lemma}

Note that a point $p$ on a non-planar curve $\delta$ lies on infinitely many trisecants of $\delta$ if and only if the projection $\pi_p$ is not generically one-to-one on $\delta$.
Thus, according to the trisecant lemma there are finitely many such projection points on $\delta$.
The following special case of a theorem of Segre's \cite{S36} shows that there are also finitely many such projection points not on $\delta$.

\begin{prop}[Segre \cite{S36}]\label{prop:trisecant}
Let $\delta$ be an irreducible non-planar curve in $\CC\PP^3$.
Then the set of points
\[ X = \setbuilder{x\in\CC\PP^3\setminus \delta}{\text{$\pi_x$ is not generically one-to-one on $\delta$}}\]
is finite.
\end{prop}

For a curve $\delta$ and a point $p$ in $\CC\PP^3$, denote the cone over $\delta$ with vertex $p$ by $C_p(\delta)$, that is,
\[ C_p(\delta) := \overline{\pi_p^{-1}(\pi_p(\delta \setminus \{p\}))}. \]
Note that if $p \notin \delta$, then \cite{H92}*{Example~3.10}
\[ C_p(\delta) = \setbuilder{q\in\CC\PP^3\setminus\{p\}}{\text{the line $pq$ intersects $\delta$}}\cup\{p\},\]
and if $p\in\delta$, then 
\[ C_p(\delta) = \setbuilder{q\in\CC\PP^3\setminus\{p\}}{\text{the line $pq$ intersects $\delta$ with multiplicity at least $2$}}\cup\{p\}.\]

\begin{lemma}\label{lem:projection2}
Let $\delta_1$ and $\delta_2$ be two irreducible curves in $\CC\PP^3$ of degree $d_1$ and $d_2$, respectively. 
Suppose $\delta_1$ is not a line, and $\delta_1 \cup \delta_2$ is non-planar.
Then there are at most $O(d_1d_2)$ points $x$ on $\delta_1$ such that $\overline{\pi_x(\delta_1\setminus\{x\})}$ and $\overline{\pi_x(\delta_2\setminus\{x\})}$ coincide, or equivalently, for which $\delta_2\subset C_x(\delta_1)$.
\end{lemma}

\begin{proof}
Let $X = \setbuilder{x\in\delta_1}{\overline{\pi_x(\delta_1\setminus\{x\})} = \overline{\pi_x(\delta_2\setminus\{x\})}}$, and let 
\[S = \delta_1\cap\bigcap_{p\in\delta_1\setminus\delta_2} C_p(\delta_2).\]
We claim that $X\setminus\delta_2 = S\setminus\delta_2$.

First, let $x\in X\setminus\delta_2$ and $p\in\delta_1\setminus\delta_2$.
If $x=p$, then clearly $x\in C_p(\delta_2)$.
Otherwise, $\pi_x(p)\in\pi_x(\delta_1\setminus\{x\})$.
Since $x\in X$, $\overline{\pi_x(\delta_1\setminus\{x\})} = \overline{\pi_x(\delta_2\setminus\{x\})}$, and since $x\notin\delta_2$, $\overline{\pi_x(\delta_2\setminus\{x\})} = \pi_x(\delta_2\setminus\{x\})$.
Therefore, $\pi_x(p)\in\pi_x(\delta_2\setminus\{x\})$, and it follows that the line $px$ intersects $\delta_2$, hence $x\in C_p(\delta_2)$.
Since $X\subseteq\delta_1$, we conclude that $x\in S\setminus\delta_2$.

Conversely, let $x\in S\setminus\delta_2$.
Then $x\in\delta_1$, and for all $p\in\delta_1\setminus\delta_2$, we have $x\in C_p(\delta_2)$.
Thus, if $x\neq p$, then the line $px$ intersects $\delta_2$.
Therefore, $\pi_x(\delta_1\setminus\{x\})\subseteq\pi_x(\delta_2\setminus\{x\})$.
Since $\delta_2$ is irreducible, the curve $\overline{\pi_x(\delta_2\setminus\{x\})}$ is irreducible.
Since $\delta_1$ is not a line, $\overline{\pi_x(\delta_1\setminus\{x\})}$ does not degenerate to a point.
Therefore, $\overline{\pi_x(\delta_1\setminus\{x\})} = \overline{\pi_x(\delta_2\setminus\{x\})}$, and $x\in X$.

Next, note that each $x\in S$ lies on infinitely many trisecants of the curve $\delta_1\cup\delta_2$.
Since $\delta_1 \cup \delta_2$ is non-planar, $S$ is finite by the trisecant lemma (Lemma~\ref{lem:trisecant}).
Therefore, $\delta_1\not\subset C_p(\delta_2)$ for some $p\in\delta_1\setminus\delta_2$.
By B\'ezout's theorem (Corollary~\ref{cor:Bezout2}), $|S|\le|\delta_1\cap C_p(\delta_2)| \le d_1\deg(C_p(\delta_2))\le d_1d_2$.
Again by B\'ezout's theorem (Corollary~\ref{cor:Bezout3}), $|\delta_1\cap\delta_2|\le d_1d_2$.
It then follows that $|X| \le |X\setminus\delta_2| + |\delta_1\cap\delta_2| = |S\setminus\delta_2| + |\delta_1\cap\delta_2| \le 2d_1d_2$.
\end{proof}

The following result is the $1$-dimensional case of a result from Ballico \cite{Ballico2003}; see also \cite{Ballico2004}*{Remark~1}.
For convenience we provide the proof of this special case.

\begin{lemma}\label{lem:projection3}
Let $\delta$ be an irreducible 
non-planar curve in $\CC\PP^3$ of degree $d$.
Then there are at most $O(d^3)$ points $x\in\CC\PP^3\setminus\delta$ such that $\pi_x$ restricted to $\delta$ is not generically one-to-one.
\end{lemma}

\begin{proof}
By Proposition~\ref{prop:trisecant}, the set
\[ X = \setbuilder{x\in\CC\PP^3\setminus \delta}{\text{$\pi_x$ is not generically one-to-one on $\delta$}}\]
is finite, and we want to show that $|X|=O(d^3)$.

Let $x\in X$.
Since $\delta$ has finitely many singularities and there are only finitely many lines through $x$ that are tangent to $\delta$, we have that for all but finitely many points $p\in\delta$, the line $px$ intersects $\delta$ in a third point, that is, $x\in C_p(\delta)$ for all $p\in\delta\setminus E_x$, for some finite subset $E_x$ of $\delta$.
Let $\delta'=\delta\setminus\bigcup_{x\in X}E_x$ and $S=\bigcap_{p\in\delta'} C_p(\delta)$.
Then clearly $X\subseteq S\setminus\delta$.
Conversely, if $x\in S\setminus\delta$, then for any $p\in\delta'$, the line $px$ intersects $\delta$ with multiplicity at least $2$.
Since only finitely many lines through $x$ can be tangent to $\delta$, it follows that for all but finitely many points $p\in\delta$, the line $px$ intersects $\delta$ in a third point, hence $x\in X$.
This shows that $X=S\setminus\delta$.

Fix distinct $p,p'\in\delta'$.
Then $X\subseteq C_p(\delta)\cap C_{p'}(\delta)$.
This intersection consists of $\delta$, the line $pp'$, and some further irreducible curves $\delta_1, \dots, \delta_k$ of total degree at most $d^2-d-1$, by B\'ezout's theorem (Corollary~\ref{cor:Bezout1}).

If some $\delta_i\subset C_p(\delta)$ for all $p\in\delta'$, then $\delta_i\subseteq S$, and since $\delta_i\cap\delta$ is finite by B\'ezout's theorem (Corollary~\ref{cor:Bezout3}), we obtain that $X$ is infinite, a contradiction.
Therefore, for each $\delta_i$ there is a point $p_i\in\delta'$ such that $\delta_i\not\subset C_{p_i}(\delta)$.
By B\'ezout's theorem (Corollary~\ref{cor:Bezout2}), $|X\cap\delta_i| \le |C_{p_i}(\delta)\cap\delta_i| \le d\deg(\delta_i)$.
It follows that $|X\setminus pp'| \le \sum_{i=1}^k|X\cap\delta_i| \le \sum_{i=1}^kd\deg(\delta_i) = O(d^3)$.

Now find a third point $p''\in\delta'$ such that $p,p',p''$ are not collinear.
As before, $|X\setminus pp''| = O(d^3)$.
Since $pp'\cap pp''$ is a singleton, it follows that $|X|=O(d^3)$.
\end{proof}

If an irreducible non-planar curve $\delta$ of degree $d$ is smooth, then by a well-known result going back to Cayley (see \cites{Bertin, LeBarz82, GP82}), the trisecant variety of $\delta$ (the Zariski closure in $\CC\PP^3$ of the union of all trisecants of $\delta$) has degree $O(d^3)$.
For $p \in \delta$, if $\pi_p$ restricted to $\delta \setminus \{p\}$ is not generically one-to-one, then $C_p(\delta)$ is a component of the trisecant variety and has degree at least $2$.
It follows that there can be at most $O(d^3)$ points $p\in\delta$ such that $\pi_p$ is not generically one-to-one on $\delta$.

However, if $\delta$ is not smooth, we are not aware of any estimates of the degree of the trisecant variety, and we thus include the proof of the weaker bound $O(d^4)$ below in Lemma~\ref{lem:projection1}, based on an argument of Furukawa \cite{F11}.
This result answers the $1$-dimensional case of a question of Ballico \cite{Ballico2004}*{Question~1}.

We say that a point $z\in\CC\PP^3$ is a \emph{vertex} of a surface $Z$ in $\CC\PP^3$ if the projection $\overline{\pi_z(Z \setminus \{z\})}$ is a planar curve with $Z$ equal to the cone $C_z(\overline{\pi_z(Z\setminus\{z\})})$.
In \cite{F11}*{Lemma~2.3}, the vertices of a surface is characterised in terms of partial derivatives.
For any $4$-tuple of non-negative integers $\vi=(i_0,i_1,i_2,i_3)$, we define $|\vi| = i_0+i_1+i_2+i_3$.
For any homogeneous polynomial $f\in\CC[x_0, x_1, x_2, x_3]$ of degree $e$, we define
\[ D_{\vi} f =
\frac{\partial^{i_0}}{\partial x_0^{i_0}}
\frac{\partial^{i_1}}{\partial x_1^{i_1}}
\frac{\partial^{i_2}}{\partial x_2^{i_2}}
\frac{\partial^{i_3}}{\partial x_3^{i_3}} f. \]
Let $D f$ be the column vector $[D_{\vi} f]_{\vi}$, where $\vi$ varies over all $4$-tuples such that $|\vi| = e-1$.
Then $D f$ is an $\binom{e+2}{3}$-dimensional vector of linear forms in $x_0,x_1,x_2,x_3$.
According to \cite{F11}*{Lemma~2.3}, $z$ is a vertex of the surface $Z$ defined by the homogeneous polynomial $f$ of degree $e$ if and only if $(D f)(z)$ is the zero vector.

\begin{lemma}\label{lem:projection1}
Let $\delta$ be an irreducible non-planar curve of degree $d$ in $\CC\PP^3$.
Then there are at most $O(d^4)$ points $x$ on $\delta$ such that $\pi_x$ restricted to $\delta \setminus \{x\}$ is not generically one-to-one.
\end{lemma}

\begin{proof}
Let $X$ be the set of points $x$ on $\delta$ such that $\pi_x$ restricted to $\delta \setminus \{x\}$ is not generically one-to-one.
Let $V \subseteq \CC[x_0, x_1, x_2, x_3]$ be the vector space of homogeneous polynomials of degree $d-2$ that vanish on $\delta$, and let $h_1,\dots,h_r$ be a basis of $V$.
Consider the matrix $A = [Dh_1, \dots, Dh_r]$.

Suppose first that $x \in X$.
Then $\deg(\overline{\pi_x(\delta\setminus\{x\}}) \le d-2$, and there exists a cone of degree $\le d-2$ with vertex $x$ containing $\delta$.
It follows that there is a polynomial $f\in V$ such that $Z_\CC(f)$ contains $\delta$.
By \cite{F11}*{Lemma~2.3}, the rank of $A(x) = [Dh_1(x), \dots, Dh_r(x)]$ is less than $r$, so each $r \times r$ minor of $A$ vanishes at $x$.
Note that each such minor defines a surface of degree at most $r$.
Conversely, if $x$ lies on all of the surfaces defined by the $r\times r$ minors of $A$, then $A(x)$ has rank less than $r$.
There then exists $f\in V$ such that $(Df)(x)$ is the zero vector.
By \cite{F11}*{Lemma~2.3}, $x$ is a vertex of $Z_\CC(f)$, which is a surface of degree at most $d-2$ and contains $\pi_x(\delta\setminus\{x\})$, so either $x$ is a singular point of $\delta$ or $x \in X$.

Since $\delta$ has at most $O(d^2$) singular points, it will follow that $X$ has at most $O(d^4)$ points if we can show that there are at most $O(d^4)$ points in
\[ \delta \cap \setbuilder{x\in\CC\PP^3}{\rank(A(x)) < r}.\]
Now $X$ is finite by the trisecant lemma (Lemma~\ref{lem:trisecant}), so $\delta$ is not a subset of all of the surfaces defined by the $r \times r$ minors of $A(x)$.
Fix one such surface $Z$ not containing $\delta$.
It has degree at most $r$, so by B\'ezout's theorem (Corollary~\ref{cor:Bezout2}), $\delta\cap Z$ has at most $dr$ points.
Since $r=O(d^3)$, the lemma follows.
\end{proof}

We have no reason to believe that the estimate $O(d^4)$ in the above lemma is best possible.

\section{Space quartics}\label{sec:quartics}

\subsection{Classification of space quartics}\label{ssec:classification}

By a \emph{space quartic}, we mean an irreducible non-planar curve of degree $4$ in $\CC\PP^3$.
Such a curve meets a generic plane in four points, and a generic quadric surface in eight points.
Since the dimension of the vector space of degree $2$ homogeneous polynomials in three variables is $10$, we can fit a quadric $Q$ through any nine points on a space quartic.
It follows by B\'ezout's theorem (Corollary~\ref{cor:Bezout2}) that the space quartic is contained in $Q$. 
A space quartic $\delta$ is said to be of the \emph{first species} if more than one quadric contains $\delta$, in which case all the quadrics containing $\delta$ form a pencil (see below).
Otherwise it is said to be of the \emph{second species}, where $\delta$ is contained in a unique quadric.
The facts collected here are well known in the sense that they were discovered in the 19th century, see \cites{S1851, F1895, B1871, W1871, R1900, T36}, but it is not easy to find recent references, and so we include some of the proofs.

The homogeneous quadratic polynomial defining a quadric $Q$ in $\FF\PP^3$ can be written as $q(x) = x^TA_Qx$, where $A_Q$ is a symmetric $4\times 4$ matrix with entries in $\FF$.
A pencil of quadrics is a collection of quadrics defined by linear combinations of two linearly independent quadratic polynomials $p$ and $q$:
$\setbuilder{Z_\FF(\lambda p+\mu q)}{\lambda,\mu\in\FF}$.

If $\delta$ is a smooth space quartic of the first species, then it is well known to be an elliptic curve \cite{S39}*{Sections~14.713 and 14.714}. 
Otherwise $\delta$ is rational~\cite{S39}*{Sections~14.717 and~14.723}, \cite{T36}*{Chapter II} and can be parametrised as $[f_1(t), f_2(t), f_3(t), f_4(t)]$ for polynomials $f_1, f_2, f_3, f_4 \in \FF[t]$ with no common factor and maximum degree $4$, with $t \in \FF$. 
Note that while $\delta$ is a curve in $\FF\PP^3$ and should be parametrised by the projective line $\FF\PP^1$, we consider $\delta$ to be parametrised by $t \in \FF$ for simplicity, where $t$ corresponds to $[t,1]\in\FF\PP^1$. 
We omit the simple proof of the following simplified parametrisation of an arbitrary rational space quartic.

\begin{prop}\label{prop:param}
Let $\delta$ be a rational space quartic in $\FF\PP^3$ parametrised by $t \in \FF$. 
After a linear fractional transformation on $\FF$ and a projective transformation on $\FF\PP^3$ if necessary, we can parametrise $\delta$ as 
\begin{equation}\label{eqn:param}
[t^4 - p, t^3 + q, t^2 - r, t + s],
\end{equation}
for some $p,q,r,s \in \FF$ such that $r \ne s^2$, $q \ne s^3$, or $p \ne s^4$.
\end{prop}

With this parametrisation, we get the following condition on when four points, counting multiplicity, on a space quartic are coplanar.

\begin{lemma}\label{lem:coplanar}
Let $\delta$ be a space quartic in $\FF\PP^3$ given by the parametrisation~\eqref{eqn:param}.
A plane intersects $\delta$ in four points parametrised by $t_1, t_2, t_3, t_4$, counting multiplicity, if and only if
\begin{equation}\label{eqn:coplanar}
F(t_1,t_2,t_3,t_4) := t_1t_2t_3t_4 + s \sum_{i<j<k} t_it_jt_k + r \sum_{i < j} t_it_j + q \sum_i t_i + p = 0.
\end{equation}
In particular, if $t_1,t_2,t_3$ are distinct, then $F(t_1,t_1,t_2,t_3)=0$ if and only if the plane through $t_1$, $t_2$, $t_3$ intersects $\delta$ only in $t_1, t_2, t_3$ and contains the tangent of $\delta$ at $t_1$.
\end{lemma}

\begin{proof}
The plane in $\FF\PP^3$ with equation $a_1 x + a_2 y + a_3 z + a_4 w = 0$ intersects $\delta$ in $t_1, t_2, t_3, t_4$ if and only if we have
\begin{equation*}
a_1 (t^4 - p) + a_2 (t^3 + q) + a_3 (t^2 - r) + a_4 (t + s) \equiv a_1(t - t_1)(t - t_2)(t - t_3)(t - t_4).
\end{equation*}
Equating coefficients, the two polynomials in $t$ are then equal identically if and only if $F = 0$, as desired.
\end{proof}

Let $f(t) := F(t,t,t,t)$ be the so-called restitution of the multilinear form $F$ \cite{D03}*{Section~1.2}, and let $g$ be the binary quartic form obtained from homogenising $f$:
\begin{equation}\label{eqn:quartic}
g(\lambda, \mu) = \lambda^4 + 4s\lambda^3 \mu + 6r\lambda^2 \mu^2 + 4q\lambda \mu^3 + p\mu^4.
\end{equation}
We call $g$ the \emph{fundamental quartic} of $\delta$. The \emph{catalecticant} of $g$ is defined to be
\begin{equation*}
\cat(g) :=
\begin{vmatrix}
1 & s & r\\
s & r & q\\
r & q & p\\
\end{vmatrix}
= pr - q^2 - ps^2 + 2qrs - r^3.
\end{equation*}
This is an invariant of $g$, in the sense that it remains unchanged under a linear change of variables
\begin{equation*}
\begin{pmatrix}
\lambda' \\ \mu'
\end{pmatrix}
=
\begin{pmatrix}
a & b \\ c & d
\end{pmatrix}
\begin{pmatrix}
\lambda \\ \mu
\end{pmatrix}
,
\end{equation*}
where $ad - bc = 1$ \cite{D03}*{Example~1.4}.
The catalecticant of a binary quartic form was discovered by Boole \cite{W08}, and generalised to binary forms of even degree by Sylvester \cite{S1851} (who coined the term).
Sylvester \cite{S1851} also showed that a generic binary form of degree $d$ is the sum of two $d$-th powers of linear forms if and only if a certain matrix does not have full rank.
We need a version of this statement that is valid for all binary forms, not only generic ones, as can be found in Kanev \cite{K99}.
The following version, formulated only for quartic forms, will be used in Section~\ref{sec:quad}.

\begin{theorem}[Sylvester]\label{thm:sylvester}
A non-zero binary quartic form 
\[g(\lambda, \mu) = 
\lambda^4 + 4c_1\lambda^3 \mu + 6c_2\lambda^2 \mu^2 + 4c_3\lambda \mu^3 + c_4\mu^4 \in \CC[\lambda, \mu]\] 
can be expressed as one of 
\[ (a \lambda + b \mu)^4, (a_1 \lambda + b_1 \mu)^4 + (a_2 \lambda + b_2 \mu)^4, (a_1 \lambda + b_1 \mu)(a_2 \lambda + b_2 \mu)^3,\]
for some $a, b, a_1, b_1, a_2, b_2 \in \CC$, where $a_1b_2 \ne b_1a_2$, if and only if $\cat(g)$ vanishes.
\end{theorem}

A more immediate application of the catalecticant is the following condition on when a rational space quartic is of the first species \cite{F1895}*{Section~6}.

\begin{lemma}\label{lem:first}
Let $\delta$ be a space quartic in $\FF\PP^3$ given by the parametrisation~\eqref{eqn:param}.
Then $\delta$ is of the first species if and only if the catalecticant of its fundamental quartic vanishes.
\end{lemma}

\begin{proof}
We prove this by considering the equations of quadrics $Q$ that contain $\delta$.
If $Q$ contains $\delta$, substituting the four polynomials $t^4-p, t^3+q, t^2-r, t+s$ for the homogeneous coordinates of $x$ into the equation $x^T A_Q x = 0$ gives a degree $8$ polynomial in $t$ that has to be identically zero.

This gives nine equations in ten variables (the entries of the symmetric $4\times 4$ matrix $A_Q=(a_{i,j})$). 
The first few equations, corresponding to the coefficients of $t^8, t^7, t^6, t^5$, are $a_{11}=0, a_{12}=0, 2a_{13}+a_{22}=0, a_{14}+a_{23}=0$.
So in fact we only have five equations in six variables:
\begin{equation*}
\begin{pmatrix}
-2r & 2s & 2 & 1 & 0 &0 \\
-2q & r & s & 0 & 1 & 0 \\
-2p & -2q & 0 & -2r & 2s & 1 \\
0 & -p & q & 0 & -r & s \\
2(pr-q^2) & -2(ps-qr) & 2qs & r^2 & -2rs & s^2
\end{pmatrix}
\begin{pmatrix}
a_{13}\\a_{14}\\a_{24}\\a_{33}\\a_{34}\\a_{44}
\end{pmatrix}
= 0.
\end{equation*}
There is always a non-trivial solution to this system, but we want to show that there are always at least two linearly independent solutions if and only if $\cat(g)=pr - q^2 - ps^2 + 2qrs - r^3 = 0$.

The nullity of the matrix is at least $2$ if and only if its rank is at most $4$, which in turn happens if and only if the six $5 \times 5$ minors all vanish. 
These six minors are
\begin{align*}
-4\cat(g)&(q^2 - pr), & 2\cat(g)&(qr - ps), & -4\cat(g)&(qs - p),\\
2\cat(g)&(r^2 - 2qs + p), & 2\cat(g)&(rs - q), & -2\cat(g)&(s^2-r),
\end{align*}
and it is impossible for all of the last factors to be equal to zero, otherwise we have $r = s^2$, $q = s^3$, and $p = s^4$.
Thus the six minors all vanish if and only if $\cat(g) = 0$, as desired.
\end{proof}

\subsection{Groups on space quartics}\label{ssec:groups}

The extremal configurations in Theorem~\ref{thm:strong} are based on group laws on certain space quartics of the first species.
We include here the reducible quartic consisting of two disjoint conic sections of a quadric, which can be viewed as the intersection of a quadric with the union of two planes (which is a degenerate quadric).
Geometrically, the group laws relate to when four points on the curve are coplanar.

A smooth space quartic of the first species in $\RR\PP^3$ is an elliptic normal curve as described in \cite{M10}*{Section~4.4.5}.
There is a group law on the points of the quartic such that four points, counting multiplicity, are coplanar if and only if they sum to the identity (see for example~\cite{S09}*{Exercise~3.10}). 
Analogous to real elliptic planar cubics, a real elliptic normal curve has one or two real connected components.
If there is one connected component, this group is isomorphic to the circle $\RR/\ZZ$; if there are two connected components, it is isomorphic to $\RR/\ZZ \times \ZZ_2$~\cite{M10}*{Section~4.4.5}.

A space quartic of the first species that is not smooth is rational with a single singular point~\cite{M10}*{Section~4.4.6}, which can be either a cusp, a crunode, or an acnode.
As discussed in \cite{M10}*{Section~4.4.6}, we can define a group law on the smooth points on such quartics.
The groups obtained are completely analogous to the groups on the smooth points of singular planar cubics.
In the case of a cuspidal space quartic, the group is isomorphic to $(\RR,+)$; for a space quartic with a crunode, it is isomorphic to the non-zero real numbers $(\RR^*,\cdot)\cong\RR \times \ZZ_2$; and an acnodal space quartic has group isomorphic to the circle group $\RR/\ZZ$.
Geometrically, in the cuspidal and acnodal cases, four smooth points, counting multiplicity, are coplanar if and only if they sum to the identity;
in the crunodal case, four smooth points, counting multiplicity, are coplanar if and only if they sum to $(0,0)$ or $(0,1)$, depending on the curve.

Lastly, given two disjoint conic sections of a quadric, we can apply a projective transformation to obtain two circles that lie on a sphere. 
Four points are coplanar if and only if they are concyclic on the sphere.
We can then define a group law on the union of the two circles completely analogous to the two concentric circles case in~\cite{LMMSSZ17}*{Section~3.2}.
Two points on each circle, counting multiplicity, are coplanar if and only if they sum to the identity, and in this case the group is isomorphic to $\RR/\ZZ \times \ZZ_2$.

We also need to know the groups on a space quartic of the first species over the complex numbers, and here the situation is analogous to complex singular planar cubics.
The singular point is either a cusp or a node (a crunode and an acnode being indistinguishable over $\CC$).
For a complex cuspidal space quartic, the group is isomorphic to $(\CC,+)$, and for a complex space quartic with a node, it is isomorphic to the non-zero complex numbers $(\CC^*,\cdot)$.

\section{Ordinary planes}\label{sec:proof}

We prove Theorem~\ref{thm:strong} in this section.
First, in Section~\ref{ssec:intermediate}, we prove the weaker Lemma~\ref{lem:intermediate}, using results from Section~\ref{sec:tools} together with a result of Green and Tao's.
This provides an alternative to Ball's tetra-grid.
We then refine Lemma~\ref{lem:intermediate} in Section~\ref{ssec:proof}, replacing the polynomial error terms by linear error terms in Lemma~\ref{lem:intermediate2}.
Finally, using the properties of space quartics from Section~\ref{sec:quartics}, we determine the precise characterisation of the possible configurations of sets with few ordinary planes as described in Theorem~\ref{thm:strong}.

\subsection{Intermediate structure theorem}\label{ssec:intermediate}

We start with the following statement of Green and Tao's intermediate structure theorem for sets with few ordinary lines~\cite{GT13}.

\begin{theorem}[Green--Tao \cite{GT13}]\label{thm:GT}
Let $P$ be a set of $n$ points in $\RR\PP^2$, spanning at most $Kn$ ordinary lines, for some $K \ge 1$.
Then we have one of the following:
\begin{enumerate}[label=\rom]
\item $P$ is contained in the union of $O(K)$ lines and an additional $O(K^6)$ points;
\item $P$ lies on the union of an irreducible conic $\sigma$ and an additional $O(K^4)$ lines, with $|P \cap \sigma| = \frac{n}{2} \pm O(K^5)$;
\item $P$ is contained in the union of an irreducible cubic and an additional $O(K^5)$ points.
\end{enumerate}
\end{theorem}

With the results in Section~\ref{sec:tools} and Theorem~\ref{thm:GT}, we can now prove the following intermediate structure theorem for ordinary planes.

\begin{lemma}[Intermediate structure theorem]\label{lem:intermediate}
Let $K\ge 1$ and suppose $n \ge CK^8$ for some sufficiently large constant $C>0$.
Let $P$ be a set of $n$ points in $\RR\PP^3$ with no three collinear.
If $P$ spans at most $Kn^2$ ordinary planes,
then we have one of the following:
\begin{enumerate}[label=\rom]
\item $P$ is contained in the union of a plane and an additional $O(K^6)$ points;\label{case:1.1}
\item $P$ is contained in the union of two irreducible conics lying on distinct planes and an additional $O(K^8)$ points, with each conic containing $\frac{n}{2} \pm O(K^8)$ points of $P$;\label{case:1.2}
\item $P$ is contained in the union of a space quartic curve and an additional $O(K^5)$ points.\label{case:1.3}
\end{enumerate}
\end{lemma}

\begin{proof}
Let $P'$ denote the set of all points $p \in P$ such that there are at most $9Kn$ ordinary planes through $p$. 
Then $|P'| \ge 2n/3$, and for any $p \in P'$, the projection $\pi_p(P \setminus \{p\})$ spans at most $9Kn$ ordinary lines. 
Applying
Theorem~\ref{thm:GT} to $\pi_p(P\setminus\{p\})$ for any $p \in P'$, we have one of three cases:
\begin{enumerate}
\item $\pi_p(P\setminus\{p\})$ is contained in the union of $O(K)$ lines and an additional $O(K^6)$ points;
\item $\pi_p(P\setminus\{p\})$ lies on the union of a conic $\sigma$ and an additional $O(K^4)$ lines with $|\pi_p(P\setminus\{p\}) \cap \sigma| = n/2 \pm O(K^5)$;
\item $\pi_p(P\setminus\{p\})$ is contained in the union of an irreducible cubic and an additional $O(K^5)$ points.
\end{enumerate}
This partitions $P'$ into a disjoint union $P_1' \cup P_2' \cup P_3'$, depending on which of the above cases we obtain.

If $|P_1'| \ge 3$, let $p_1, p_2, p_3$ be three distinct points in $P_1'$. 
Then apart from $O(K^6)$ points, $P$ is contained in the intersection of the union of $O(K)$ planes through $p_1$, the union of $O(K)$ planes through $p_2$, and the union of $O(K)$ planes through $p_3$.

Since $p_1, p_2, p_3$ are not collinear, if $\Pi_i$ is a plane through $p_i$, then $P \cap \Pi_1 \cap \Pi_2 \cap \Pi_3$ is contained in a line, which contains at most two points of $P$ except when $\Pi_1 = \Pi_2 = \Pi_3$ is the plane through $p_1, p_2, p_3$.
Thus we have $P$ lying in a plane except for at most $O(K^6) + O(K^2) = O(K^6)$ points, giving Case~\ref{case:1.1}.

Next suppose $|P_2'| \ge 3n/5$, and let $p_1, p_2, p_3$ be three distinct points in $P_2'$. 
Then for each $i = 1, 2, 3$, there exist a quadric cone $C_i$ with vertex $p_i$ and planes $\{\Pi_{i,j} : j \in J_i\}$ through $p_i$ with $|J_i| = O(K^4)$,
such that $P \subset C_i \cup \bigcup_{j \in J_i} \Pi_{i,j}$ with $|P \cap C_i| = n/2 \pm O(K^5)$.
So all but at most $O(K^8)$ points of $P$ lie either on the intersection $C_1 \cap C_2 \cap C_3$, one of the $O(K^4)$ conics $C_i \cap \Pi_{i',j}$ for $i \ne i'$, or the plane $\Pi$ through $p_1, p_2, p_3$.

It is well known (and easy to deduce from B\'ezout's theorem (Corollary~\ref{cor:Bezout1})) that the intersection of two quadrics is either an irreducible space quartic, a twisted cubic and a line, or conics and lines.
We claim that any component $\delta$ of the intersection $C_1 \cap C_2 \cap C_3$ that is a twisted cubic or a space quartic cannot contain more than $O(K^4)$ points of $P$.
Choose a point $p \in P_2' \setminus (C_1 \cap C_2 \cap C_3)$ such that the projection $\pi_p$ restricted to $\delta$ is generically one-to-one.
Such a $p$ exists since $|P_2' \setminus (C_1 \cap C_2 \cap C_3)| \ge 3n/5 - (n/2 + O(K^5))$ and by Lemma~\ref{lem:projection3} there are only $O(1)$ exceptional points.
Then $\pi_p(\delta)$ is an irreducible planar cubic or quartic containing more than $O(K^4)$ points of $P$, contradicting $p \in P_2'$.
So the components of $C_1 \cap C_2 \cap C_3$ which contain more than $O(K^4)$ points of $P$ must all be conics.

No plane $\Pi'$ can contain more than $n/2 + O(K^5)$ points of $P$, otherwise projecting from $p' \in P_2' \cap \Pi'$ would give a line containing more than $n/2 + O(K^5)$ points in the projection, contradicting $p' \in P_2'$.
So choose a fourth point $p_4 \in P_2' \setminus \Pi$.
As before, $P$ is contained in the union of a quadric cone $C_4$ with vertex $p_4$ and $O(K^4)$ planes through $p_4$.
Since $p_4 \notin \Pi$, if $\Pi$ contains more than $O(K^4)$ points of $P$, all but at most $O(K^4)$ points of $P \cap \Pi$ must lie on the conic $C_4 \cap \Pi$.
We then have that all but at most $O(K^8)$ points of $P$ lie on $O(K^4)$ conics, and without loss of generality we can assume each conic contains more than $O(K^4)$ points.
Let $\Sigma$ be the set of such conics.
(Note that the same argument shows that if a plane $\Pi'$ contains a conic $\sigma \in \Sigma$, then all but at most $O(K^4)$ points of $P \cap \Pi'$ lie on $\sigma$.)
We show that $|\Sigma| = 2$, and that for each $\sigma \in \Sigma$, $|P \cap \sigma| = n/2 \pm O(K^8)$, thus giving Case~\ref{case:1.2}.

Let $\sigma_1$ be the conic in $\Sigma$ with the most points of $P$, and let $\Pi_1$ be the plane in which $\sigma_1$ lies.
Since no plane contains more than $n/2 + O(K^5)$ points of $P$, we have that $|P \setminus \Pi_1| \ge n/2 - O(K^5)$.
Let $\sigma_2$ be the conic in $\Sigma_1 \setminus \{\sigma_1\}$ with the most points of $P \setminus \Pi_1$, and let $\Pi_2$ be the plane in which $\sigma_2$ lies.
Then $|P \cap \sigma_1| \ge |P \cap \sigma_2| \ge \Omega(n/K^4)$.
Note that $\Pi_1\neq\Pi_2$, as no plane contains more than $n/2+O(K^5)$ points of $P$.
Suppose there exists $q \in P_2' \setminus (\Pi_1 \cup \Pi_2)$, so that (by B\'ezout's theorem (Corollary~\ref{cor:Bezout2})) $\sigma_1$ and $\sigma_2$ must both lie on the same quadric cone $C$ with vertex $q$.
Since $\sigma_1 \cup \sigma_2$ is the intersection of $\Pi_1 \cup \Pi_2$ and $C$, there can only be at most two such points by Proposition~\ref{prop:cones}.
Therefore, all but at most $O(K^4)$ points of $P_2'$ are contained in $\sigma_1 \cup \sigma_2$.

Without loss of generality, suppose $|P_2' \cap \sigma_1| \ge 3n/10 - O(K^4) = \Omega(n)$.
By Proposition~\ref{prop:cones} again, there exists at most two points in $P_2' \cap \sigma_1$ such that their quadric cones intersect in $\sigma_2$ and $\sigma'$ for some $\sigma' \in \Sigma \setminus \{\sigma_1, \sigma_2\}$.
We can then choose a $q' \in P_2' \cap \sigma_1$ such that the only conic in $\Sigma$ the quadric cone with vertex $q'$ contains is $\sigma_2$.
In particular, this means that $|P \cap \sigma_2| = n/2 \pm O(K^8)$.
But then we also have $|P_2' \cap \sigma_2| = \Omega(n)$.
Repeating the argument on $\sigma_2$ shows that $|P \cap \sigma_1| = n/2 \pm O(K^8)$ as well.

The remaining case is when $|P_3'| > 2n/3 - 3n/5 - 3 = \Omega(n)$.
Let $p$ and $p'$ be two distinct points in $P_3'$.
Then apart from $O(K^5)$ points, we have $P$ lying mostly on the intersection $\delta$ of two cubic cones, which is a curve with irreducible components $\delta_i$ of total degree $9$ by B\'ezout's theorem (Corollary~\ref{cor:Bezout1}).

Let $\delta_1$ be a component for which $|P_3' \cap \delta_i|$ is maximal.
Then $|P_3' \cap \delta_1| = \Omega(n)$.
Projecting from any $q \in P_3' \cap \delta_1$, we get that $\overline{\pi_q(\delta_1 \setminus \{q\})}$ is an irreducible cubic, and so $\delta_1$ must be non-planar.
By Lemma~\ref{lem:projection1}, all but $O(1)$ points $q'$ on $\delta_1$ are such that the projection $\pi_{q'}$ restricted to $\delta_1 \setminus \{q'\}$ is generically one-to-one.
We can thus choose such a $q' \in P_3' \cap \delta_1$ so that $\pi_{q'}$ projects $\delta_1 \setminus \{p_1\}$ generically one-to-one onto an irreducible cubic.
The component $\delta_1$ must then be a space quartic.

Now suppose there exists a component $\delta_2$ containing more than $O(K^5)$ points of $P$.
For any $q \in P_3' \cap \delta_1$, the cone $C_q(\delta_1)$ over $\delta_1$ has to contain $\delta_2$ by B\'ezout's theorem (Corollary~\ref{cor:Bezout2}).
Since $\delta_1$ is non-planar, this contradicts Lemma~\ref{lem:projection2}.
So we have all but at most $O(K^5)$ points of $P$ lying on a single space quartic, giving Case~\ref{case:1.3}.
\end{proof}

\subsection{Proof of the full structure theorem}\label{ssec:proof}

To get a more precise description of the structure of sets spanning few ordinary planes, we need a more precise description of sets that lie on certain cubic curves and span few ordinary lines.
The following two lemmas are~\cite{GT13}*{Lemma~7.4} and~\cite{GT13}*{Lemma~7.2}, respectively.

\begin{lemma}[Green--Tao \cite{GT13}]\label{lem:7.4}
Let $P$ be a set of $n$ points in $\RR\PP^2$ spanning at most $Kn$ ordinary lines, and suppose $n \ge CK$ for some sufficiently large absolute constant $C$.
Suppose all but $K$ points of $P$ lie on the union of an irreducible conic $\sigma$ and a line $\l$, with $n/2 \pm O(K)$ points of $P$ on each of $\sigma$ and $\l$.
Then up to a projective transformation, $P$ differs in at most $O(K)$ points from the vertices of a regular $m$-gon and the $m$ points at infinity corresponding to the diagonals of the $m$-gon, for some $m = n/2 \pm O(K)$.
\end{lemma}

\begin{lemma}[Green--Tao \cite{GT13}]\label{lem:7.2}
Let $P$ be a set of $n$ points in $\RR\PP^2$ spanning at most $Kn$ ordinary lines, and suppose $n \ge CK$ for some sufficiently large absolute constant $C$.
Suppose all but $K$ points of $P$ lie on an irreducible cubic $\gamma$.
Then $P$ differs in at most $O(K)$ points from a coset of a subgroup of $\gamma^*$, the smooth points of $\gamma$.
In particular, $\gamma$ is either an elliptic curve or an acnodal cubic.
\end{lemma}

We also need two further technical lemmas.
The following are \cite{GT13}*{Corollary~7.6} and \cite{GT13}*{Lemma~7.7}, respectively.

\begin{lemma}[Green--Tao \cite{GT13}]\label{cor:7.6}
Let $X_{2m}$ be the vertex set of a regular $m$-gon centred at the origin of the Euclidean plane, together with the $m$ points at infinity corresponding to the diagonals of the $m$-gon.
Let $p$ be a point not on the line at infinity, not the origin, and not a vertex of the $m$-gon.
Then at least $2m - O(1)$ of the $2m$ lines joining $p$ to a point of $X_{2m}$ do not pass through any further point of $X_{2m}$.
\end{lemma}

\begin{lemma}[Green--Tao \cite{GT13}]\label{lem:7.7}
Let $\gamma^*$ be an elliptic curve or the smooth points of an acnodal cubic curve.
Let $X$ be a coset of a finite subgroup of $\gamma^*$ of size $n$, where $n$ is greater than a sufficiently large absolute constant.
If $p\in\RR\PP^2\setminus\gamma^*$, then there are at least $n/1000$ lines through $p$ that pass through exactly one point in $X$.
\end{lemma}

Using the above four lemmas, we reduce the polynomial error terms in Lemma~\ref{lem:intermediate} to linear errors $O(K)$.
Since this refinement relies only on Green and Tao's results in~\cite{GT13}*{Section 7}, our proof will be similar to Ball's proof in~\cite{B17}.

\begin{lemma}\label{lem:intermediate2}
Let $K\ge 1$ and suppose $n \ge CK^8$ for some sufficiently large constant $C>0$.
Let $P$ be a set of $n$ points in $\RR\PP^3$ with no three collinear.
If $P$ spans at most $Kn^2$ ordinary planes,
then $P$ differs in at most $O(K)$ points from one of the following:
\begin{enumerate}[label=\rom]
\item A subset of a plane;\label{case:2.1}
\item A subset of two disjoint irreducible conics lying on distinct planes, each containing $\frac{n}{2} \pm O(K)$ points;\label{case:2.2}
\item A subset of a space quartic.\label{case:2.3}
\end{enumerate}
\end{lemma}

\begin{proof}
Let $P'$, $P_1'$, $P_2'$, and $P_3'$ be as in the proof of Lemma~\ref{lem:intermediate}.

If $|P_1'| \ge 3$, we are in Case~\ref{case:1.1} of Lemma~\ref{lem:intermediate}, and all but at most $O(K^6)$ points of $P$ lie in a plane $\Pi$.
Let $k:=|P\setminus\Pi|$.
Then for a fixed point $p\in P\setminus\Pi$, there are $\binom{n-k}{2}$ planes through $p$ and two points of $P\cap\Pi$, of which at most $k-1$ are not ordinary.
Therefore, there are at least $k(\binom{n-k}{2}-k+1)$ ordinary planes.
Since this is at most $Kn^2$ and $n \ge CK^8 > k = O(K^6)$ for sufficiently large $C$, we obtain that $k=O(K)$.

If $|P_2'| \ge 3n/5$, we are in Case~\ref{case:1.2} of Lemma~\ref{lem:intermediate}, and all but at most $O(K^8)$ points of $P$ lie on the union of two conics $\sigma_1 \cup \sigma_2$.
Let $S=P\setminus(\sigma_1\cup\sigma_2)$.
Let $\Pi_i$ be the plane supporting $\sigma_i$, $i=1,2$.
For any $p\in S\cap\Pi_1$ except at most two points also on $\Pi_2$, the projection $\pi_p$ maps $\sigma_1$ to a line and $\sigma_2$ to a conic.
For any $p\in S\setminus\Pi_1$ except at most two points, the projection $\pi_p$ maps $\sigma_1$ and $\sigma_2$ to distinct conics (by Proposition~\ref{prop:cones}).
Hence, for all but at most four points $p\in S$, there are at most four points $x\in P\cap\sigma_1$ for which $\pi_p(x)\in\pi_p(\sigma_2)$ and at most four points $x\in P\cap\sigma_2$ for which $\pi_p(x)\in\pi_p(\sigma_1)$.
Therefore, for any $q\in P_2' \cap (\sigma_1 \cap \sigma_2)$ except at most $4|S|+4|S|$ points, we have that $\pi_q(S)$ is disjoint from the line and the conic onto which all but at most $O(K^8)$ points of $P$ map.
Such a $q$ exists as $|P_2'|\ge 3n/5$.
By Lemma~\ref{lem:7.4}, the set $\pi_q(P\setminus\{q\})$ differs from a set $X$ projectively equivalent to a regular $m$-gon and the $m$ points at infinity corresponding to the diagonals of the $m$-gon in at most $O(K^8)$ points, where $m=n/2\pm O(K^8)$.
By Lemma~\ref{cor:7.6}, there are at least $n/2-O(K^8)$ ordinary lines through a fixed point of $\pi_q(S)$ and a point of $\pi_q(P\setminus(S\cup\{q\}))$.
Thus, there at least $|S|(n/2-O(K^8))$ ordinary lines.
Since there are at most $9Kn$ ordinary lines and $n \ge CK^8 > |S|=O(K^8)$ for sufficiently large $C$, we obtain that $|S|=O(K)$.
The same argument shows that $|\pi_q(P\setminus\{q\})\setminus X| =O(K)$, hence $|P\cap\sigma_i|\le m+O(K)$, $i=1,2$.

It remains to show that $|X\setminus\pi_q(P\setminus\{q\})|=O(K)$.
Note that through any point $y \in X$, there are at least $m/2 - 1$ lines through $y$ and two more points of $X$.
By removing a point from $X$, we thus create at least $m/2-O(K)$ ordinary lines.
Therefore, there are at least $|X\setminus\pi_q(P\setminus\{q\})|(n/4-O(K^8))$ ordinary lines.
Since this is at most $9Kn$ and $n \ge CK^8$ for sufficiently large $C$, we obtain that $|X\setminus\pi_q(P\setminus\{q\})|=O(K)$.

Finally, if $|P_3'| > 2n/3 - 3n/5 - 3 = \Omega(n)$, we are in Case~\ref{case:1.3} of Lemma~\ref{lem:intermediate}, and all but at most $O(K^5)$ points of $P$ lie on a space quartic $\delta$.
By Lemma~\ref{lem:projection3}, the projection from all but finitely many points $p\in P\setminus\delta$ maps $\delta$ generically one-to-one onto a planar quartic, which has at most three singular points.
Thus, there are at most six points $x\in\delta$ such that $px$ intersects $\delta$ again.
Choose a point $q\in P_3' \cap \delta$ that is not one of these at most $6|P\setminus\delta|$ points.
Such a $q$ exists as $|P_3'| = \Omega(n)$.
Then, if we project from $q$, each point in $P\setminus\delta$ is projected onto a point not on the cubic $\overline{\pi_q(\delta\setminus\{q\})}$.
By Lemma~\ref{lem:7.2}, $\pi_q(P\setminus\{q\})$ differs in at most $O(K^5)$ points from a coset of a subgroup of an elliptic curve or the smooth points of an acnodal cubic.
By Lemma~\ref{lem:7.7}, there are at least $|P\setminus\delta|(n/1000-|P\setminus\delta|)$ ordinary lines.
Since this is at most $9Kn$ and $n \ge CK^8 > |P\setminus\delta|=O(K^5)$ for sufficiently large $C$, we obtain that $|P\setminus\delta|=O(K)$.
\end{proof}

We now show that if we are in Case~\ref{case:2.2} of Lemma~\ref{lem:intermediate2}, then there is a quadric containing both conics.

\begin{lemma}\label{lem:conics}
Let $K\ge 1$ and suppose $n \ge CK^8$ for some sufficiently large constant $C>0$.
Let $P$ be a set of $n$ points in $\RR\PP^3$ with no three collinear, spanning at most $Kn^2$ ordinary planes,
Suppose $P$ has $n/2 \pm O(K)$ points on each of two disjoint conics $\sigma_1$ and $\sigma_2$ lying on two distinct
planes $\Pi_1$ and $\Pi_2$, respectively.
Then there exists an irreducible quadric that contains $\sigma_1\cup\sigma_2$.
\end{lemma}

\begin{proof}
Let $P'$ be as in the proofs of Lemmas~\ref{lem:intermediate} and~\ref{lem:intermediate2} above.
We convert the problem to one in Euclidean geometry, by identifying $\RR\PP^3$ with the Euclidean affine space $\RR^3$ together with a projective plane at infinity.
We apply a projective transformation such that the planes $\Pi_1$ and $\Pi_2$ become parallel, and such that $\sigma_1$ is a circle. 
It then suffices to show that $\sigma_2$ is a circle as well, as $\sigma_1\cup\sigma_2$ is then contained in a circular cylinder.

Choose $p_i \in P' \cap \sigma_i$, $i=1,2$,
and consider the projection $\pi_i:=\pi_{p_i}$ to be onto the plane $\Pi_{3-i}$.
Then by Lemma~\ref{lem:intermediate2}, $\pi_1$ projects all but at most $O(K)$ points of $P$ onto the line at infinity and a disjoint conic on $\Pi_2$. 
Since the conic $\sigma_2=\pi_1(\sigma_2)$ is disjoint from the line at infinity, it is an ellipse.

Now let $p_2 \in P' \cap \sigma_2$, and consider the projection $\pi_2$.
By Lemma~\ref{lem:7.4}, $\pi_2(P \setminus \{p_2\})$ differs in at most $O(K)$ points from the vertices of a regular $m$-gon and the $m$ points at infinity corresponding to the diagonals of the $m$-gon, for some $m = n/2 \pm O(K)$.
In particular, $P \cap \sigma_1$ differs in at most $O(K)$ points from a regular $m$-gon.

Therefore, $\pi_1(P \cap \sigma_1 \setminus \{p_1\})$ differs in at most $O(K)$ points from the points at infinity corresponding to the diagonals of the regular $m$-gon on $\sigma_1$,
which are also the points at infinity corresponding to the tangents to $\sigma_1$ at the vertices of the $m$-gon.
By Lemma~\ref{lem:7.4} again, $\pi_1(P \cap \sigma_2)$ differs in at most $O(K)$ points from an $m$-gon on the ellipse $\sigma_2$, projectively equivalent to a regular $m$-gon.  

It easily follows that all but at most $O(K)$ of the tangent lines to $\sigma_2$ at the vertices of the $m$-gon are ordinary lines and so all but $O(K)$ must be points at infinity of $\pi_1(P \cap \sigma_1 \setminus \{p_1\})$. 
Let $a,b,c$ be three consecutive vertices of the $m$-gon on $\sigma_2$ such that $a,b,c\in\pi_1(P \cap \sigma_2)$.
Then the point $d$ where the tangents at $a$ and $c$ intersect, forms an isosceles triangle with $a$ and $c$, and we have $|ad| = |cd|$.
But this can only happen if $d$ lies on one of the axes of symmetry of $\sigma_2$.
Since $n$ is sufficiently large depending on $K$, we can find many triples of consecutive vertices of the $m$-gon on $\sigma_2$, and we get a contradiction unless $\sigma_2$ is a circle.
\end{proof}

The following lemma shows that if we are in Case~\ref{case:2.3} of Lemma~\ref{lem:intermediate2}, then the space quartic is either elliptic or rational of the first species.

\begin{lemma}\label{lem:quartics}
Let $K\ge 1$ and supose $n\ge CK^8$ for some sufficiently large constant $C>0$.
Let $P$ be a set of $n$ points in $\RR\PP^3$ with no three collinear, spanning at most $Kn^2$ ordinary planes.
Suppose all but at most $O(K)$ points of $P$ lie on a rational space quartic $\delta$ given by the parametrisation \eqref{eqn:param}.
Then $\delta$ is of the first species.
\end{lemma}

\begin{proof}
Let $p_\alpha \in \delta$ be parametrised by $\alpha\in\RR$ in such a way that the point of $\delta$ not corresponding to any $\alpha\in\RR$ is not in $P$.
As before, let $P'$ denote the set of all points $p\in P\cap\delta$ with at most $9Kn$ ordinary planes through $p$.
Since $P$ spans at most $Kn^2$ ordinary planes, we have $|P'|\ge 2n/3-O(K)$.
Let $\pi_\alpha$ be the projection map from $p_\alpha$ onto a plane not containing $p_\alpha$.
By Lemma~\ref{lem:projection1}, we can choose $p_\alpha \in P'$ such that all but at most $O(K)$ points of $\pi_\alpha(P \setminus \{p_\alpha\})$ lie on a cubic curve $\gamma_\alpha$. 
By Lemma~\ref{lem:7.2}, this set differs in at most $O(K)$ points from a coset of $\gamma_\alpha$, and $\gamma_\alpha$ is acnodal (since $\delta$ is rational).
Since $n$ is sufficiently large, there exist three distinct points $p_A, p_B, p_C \in P'$ such that for $\Omega(n)$ many $p_\alpha \in P'$, the projected points $\pi_\alpha(p_A),\pi_\alpha(p_B),\pi_\alpha(p_C)$ are consecutive elements in the coset given by Lemma \ref{lem:7.2}.

Let $\oplus_\alpha$ denote the group operation on $\gamma_\alpha$ so that we have $\pi_\alpha(p_A) \oplus_\alpha \pi_\alpha(p_C) = 2\pi_\alpha(p_B)$.
By considering the geometric definition of $\oplus_\alpha$ we obtain that if $p_\beta$ is the fourth point of intersection between $\delta$ and the plane through $p_A, p_C, p_\alpha$, and $p_{\beta'}$ the fourth point of intersection between $\delta$ and the plane through $p_B, p_B, p_\alpha$ (that is, containing the tangent line of $\delta$ at $p_B$ and passing through $p_\alpha$), then $\beta = \beta'$.
Equivalently, by Lemma~\ref{lem:coplanar}, we have 
\begin{equation*}
F(A, C, \alpha, \beta) = 0 = F(B, B, \alpha, \beta).
\end{equation*}
Since $F$ is a polynomial and the above is true for sufficiently many $\alpha$ (since $n$ is sufficiently large), it holds for all $\alpha \in \RR$.

Note that
\begin{equation*}
F(A, C, \alpha, \beta) = 
\begin{pmatrix}
\alpha \beta, & \alpha + \beta, & 1
\end{pmatrix}
\begin{pmatrix}
1 & s & r\\
s & r & q\\
r & q & p\\
\end{pmatrix}
\begin{pmatrix}
AC\\
A+C\\
1\\
\end{pmatrix}
,
\end{equation*}
with a similar expression for $F(B,B,\alpha,\beta)$.
Since $\begin{pmatrix} AC, & A+C, & 1 \end{pmatrix}$ and $\begin{pmatrix} B^2, & 2B, & 1 \end{pmatrix}$ are linearly independent, the set of vectors
\begin{equation*}
\left\lbrace 
\begin{pmatrix}
\alpha \beta, & \alpha + \beta, & 1
\end{pmatrix}
\begin{pmatrix}
1 & s & r\\
s & r & q\\
r & q & p\\
\end{pmatrix}
: \alpha \in \RR
\right\rbrace
\end{equation*}
lie in a $1$-dimensional linear subspace of $\RR^3$.
If the catalecticant of the fundamental quartic of $\delta$, which is the determinant $pr - q^2 - ps^2 + 2qrs -r^3$ of the above $3\times 3$ matrix, is non-zero, then
\begin{equation*}
\left\lbrace 
\begin{pmatrix}
\alpha \beta, & \alpha + \beta, & 1
\end{pmatrix}
: \alpha \in \RR
\right\rbrace
\end{equation*}
also lies in a $1$-dimensional subspace of $\RR^3$, in which case both $\alpha \beta$ and $\alpha + \beta$ are constants depending only on $\delta$, say $\alpha \beta = c_1$ and $\alpha + \beta = c_2$.
But then $\alpha$ is a root of the fixed quadratic equation $x^2 - c_2x + c_1 = 0$, a contradiction.
Hence, the catalecticant vanishes, and $\delta$ is of the first species by Lemma~\ref{lem:first}.
\end{proof}

From Lemmas~\ref{lem:intermediate2},~\ref{lem:conics}, and~\ref{lem:quartics}, we see that up to at most $O(K)$ points, the set $P$ lies on a plane, two disjoint conic sections of an irreducible quadric (which by applying a projective transformation if necessary we can assume to be two disjoint circles on a sphere), or a space quartic of the first species.

Noting from Section~\ref{ssec:groups} that only elliptic and acnodal space quartics admit arbitrarily large finite subgroups,
Lemmas~\ref{lem:conics2} and~\ref{lem:quartics2}, which are proved in a similar way as~\cite{LMMSSZ17}*{Lemmas~5.5 and~5.6}, respectively, complete the proof of Theorem~\ref{thm:strong}.
Note that these are analogues of Lemmas~\ref{lem:7.4} and~\ref{lem:7.2} for ordinary planes.

\begin{lemma}\label{lem:conics2}
Let $P$ be a set of $n$ points in $\RR\PP^3$ with no three collinear, spanning at most $Kn^2$ ordinary planes, and suppose $n \ge CK$ for some sufficiently large absolute constant $C$.
Suppose all but at most $K$ points of $P$ lie on two disjoint circles, with $n/2 \pm O(K)$ points of $P$ on each circle.
Then, up to a projective transformation, $P$ differs in at most $O(K)$ points from a prism or an antiprism.
\end{lemma}

\begin{lemma}\label{lem:quartics2}
Let $P$ be a set of $n$ points in $\RR\PP^3$ with no three collinear, spanning at most $Kn^2$ ordinary planes, and suppose $n \ge CK$ for some sufficiently large absolute constant $C$.
Suppose all but at most $K$ points of $P$ lie on a space quartic $\delta$ of the first species.  
Then $P$ differs in at most $O(K)$ points from a coset of a subgroup of $\delta^*$, the smooth points of $\delta$.
In particular, $\delta$ is either an elliptic or acnodal space quartic.
\end{lemma}

\section{Coplanar quadruples}\label{sec:quad}

The following is a special case of Raz, Sharir, and De Zeeuw's $4$-dimensional generalisation of the Elekes-Szab\'o theorem~\cite{RSZ16}, which we use to prove Theorem~\ref{thm:quad}.

\begin{theorem}\label{thm:ES}
Let $F \in \CC[t_1, t_2, t_3, t_4]$ be irreducible of degree $d$, with no $\partial F / \partial t_i$ identically zero.
Then either for all $A \subset \CC$ with $|A| = n$, we have $|Z_\CC(F) \cap A^4| = O(n^{8/3})$, 
or there exists a $2$-dimensional subvariety $Z_0 \subset Z_\CC(F)$ such that for any $(s_1, s_2, s_3, s_4) \in Z_\CC(F) \setminus V_0$, there exist open neighbourhoods $U_i$ of $s_i$ and injective analytic functions $\phi_i: U_i \rightarrow \CC$ such that
\begin{equation*}
F(t_1, t_2, t_3, t_4) = 0 \Longleftrightarrow \phi_1(t_1) + \phi_2(t_2) + \phi_3(t_3) + \phi_4(t_4) = 0,
\end{equation*}
for all $(t_1, t_2, t_3, t_4) \in U_1 \times U_2 \times U_3 \times U_4$.
\end{theorem} 

\begin{proof}[Proof of Theorem~\ref{thm:quad}]
We first show that if an $n$-point set $P$ on a rational space quartic $\delta$ in $\CC\PP^3$ does not span $O(n^{8/3})$ coplanar quadruples, then $\delta$ is of the first species.

By Proposition~\ref{prop:param}, we may assume $\delta$ is given by the parametrisation~\eqref{eqn:param}.
Lemma~\ref{lem:coplanar} says four points parametrised by $t_1, t_2, t_3, t_4$ are coplanar if and only if \eqref{eqn:coplanar} holds.
It is clear that $F$ is not independent of any $t_i$ (otherwise $\delta$ would be planar).
Since $P$ does not span $O(n^{8/3})$ coplanar quadruples, Theorem~\ref{thm:ES} gives the existence of injective analytic $\phi_1, \phi_2, \phi_3, \phi_4$ such that $F = 0$ if and only if $\phi_1(t_1) + \phi_2(t_2) + \phi_3(t_3) + \phi_4(t_4) = 0$.
In particular, on the hypersurface $F=0$, we can express $t_4 = t_4 (t_1, t_2, t_3)$ as a function of $t_1, t_2, t_3$:
\begin{equation}\label{t4}
t_4(t_1,t_2,t_3) = -\dfrac{s t_{1} t_2t_3 + r(t_{1}t_2 + t_{1}t_3 + t_{2}t_3) + q(t_{1} + t_{2} + t_3) + p}{ t_{1} t_{2}t_3 + s(t_{1}t_2 + t_{1}t_3 + t_2t_3) + r(t_{1} + t_2 + t_3) + q}.
\end{equation}
We thus have for all $(t_1,t_2,t_3)\in U_1\times U_2\times U_3$ that
\begin{equation*}
\phi_1(t_1) + \phi_2(t_2) + \phi_3(t_3) + \phi_4(t_4(t_1,t_2,t_3)) = 0.
\end{equation*}
Partial differentiation with respect to $t_i$ for $i = 2,3$ gives
\begin{equation*}
\phi_i'(t_i) + \phi_4'(t_4) \frac{\partial t_4}{\partial t_i} = 0,
\end{equation*}
and so the quotient $\left( \partial t_4 / \partial t_2 \right) / \left( \partial t_4 / \partial t_3 \right)$ is independent of $t_1$.
The numerator of the partial derivative of this quotient with respect to $t_1$ is thus identically zero.
If we substitute \eqref{t4} into \[ \frac{\partial}{\partial t_1}\left(\frac{\partial t_4}{\partial t_2} \middle/ \frac{\partial t_4}{\partial t_3}\right) = 0,\]
we obtain (with the help of a computer algebra system such as SageMath)
\begin{equation*}
\frac{(pr - q^2 - ps^2 + 2qrs - r^3)F(t_1, t_1, t_2, t_3)(t_3 - t_2)}{h(t_1,t_2)^2} = 0,
\end{equation*}
where
\begin{align*}
h(t_1, t_2) &= (s^{2} - r) t_{1}^{2} t_{2}^{2} + (r s - q) t_1 t_2(t_{1} + t_{2}) + (r^{2} -qs) (t_{1}^{2} + t_2^2)\\
&\mathrel{\phantom{=}} \mbox{} + (r^{2} - p) t_{1} t_{2} + (q r - p s) (t_{1} + t_{2}) + q^{2} - p r.
\end{align*}
Since $t_1, t_2, t_3$ are arbitrary in $(U_1\times U_2\times U_3)\setminus (Z_{\CC}(h)\times\CC)$, we obtain that the catalecticant $pr - q^2 - ps^2 + 2qrs - r^3$ vanishes.
By Lemma \ref{lem:first}, the rational space quartic $\delta$ is thus of the first species, as desired.

For the converse, suppose that $\delta$ is of the first species.
Then, as is well known (and explicitly demonstrated in Lemma~\ref{lem:group} below), the smooth points $\delta^*$ carry a group structure such that four points are coplanar if and only if their sum in the group is the identity.
If $\delta$ is nodal, then the group is isomorphic to the non-zero complex numbers under multiplication $(\CC^*,\cdot)$.
If $\delta$ is cuspidal, then the group is isomorphic to the complex numbers under addition $(\CC,+)$.
In both groups it is trivial to find $n$ elements such that there are $\Theta(n^3)$ quadruples of distinct elements that sum to $0$.
In the multiplicative case, we can take the $n$-th roots of unity, and in the additive case, we can take the $n$ integers closest to $0$.
\end{proof}

If a rational space quartic $\delta$ of the first species is given by the parametrisation~\eqref{eqn:param}, we can also find the $\phi_i$'s in Theorem~\ref{thm:ES} explicitly.
For convenience, in the next lemma we identify $\CC\PP^1$ with the affine line $\CC$ together with a point $\infty$ at infinity.
\begin{lemma}\label{lem:group}
Let $\delta$ be a rational space quartic of the first species in $\CC\PP^3$ given by the parametrisation~\eqref{eqn:param}.

If $\delta$ is nodal, then there exists a parametrisation $\phi\colon\CC\PP^1\to\delta$ such that $\phi(0)=\phi(\infty)$ is the node of $\delta$, and any four points $\phi(t_1), \phi(t_2),\phi(t_3),\phi(t_4)$ on $\delta\setminus\{\phi(0)\}$ are coplanar if and only if $t_1t_2t_3t_4=1$.

If $\delta$ is cuspidal, then there exists a parametrisation $\phi\colon\CC\PP^1\to\delta$ such that $\phi(\infty)$ is the node of $\delta$, and any four points $\phi(t_1), \phi(t_2),\phi(t_3),\phi(t_4)$ on $\delta\setminus\{\phi(\infty)\}$ are coplanar if and only if $t_1+t_2+t_3+t_4=0$.
\end{lemma}

\begin{proof}
Since $\delta$ is of the first species, the catalecticant of its fundamental quartic given by \eqref{eqn:quartic} vanishes by Lemma~\ref{lem:first}.
Applying Theorem~\ref{thm:sylvester} and dehomogenising, we can write $f(t)$ as one of
\[(a t + b)^4, (a_1 t + b_1)^4 + (a_2 t + b_2)^4, (a_1 t + b_1)(a_2 t + b_2)^3,\]
for some $a, b, a_1,b_1,a_2,b_2 \in \CC$, where $a_1b_2 \ne b_1a_2$.

We cannot have $f(t) = (at + b)^4$, otherwise we have $r = s^2$, $q = s^3$, and $p = s^4$, contradicting the parametrisation~\eqref{eqn:param}.

If $f(t) = (a_1 t + b_1)^4 + (a_2 t + b_2)^4$, then its so-called polarisation \cite{D03}*{Section~1.2} will be equal to
\begin{equation*}
F(t_1,t_2,t_3,t_4) = \prod_{i=1}^4 (a_1 t_i + b_1) + \prod_{i=1}^4 (a_2 t_i + b_2).
\end{equation*}
So $F = 0$ if and only if
\begin{equation*}
\frac{a_1 t_1 + b_1}{a_2 t_1 + b_2} \cdot \frac{a_1 t_2 + b_1}{a_2 t_2 + b_2} \cdot \frac{a_1 t_3 + b_1}{a_2 t_3 + b_2} \cdot \frac{a_1 t_4 + b_1}{a_2 t_4 + b_2} = -1.
\end{equation*}
This is the nodal case, with the node corresponding to $\phi(t)$ where the fractional linear transformation $(a_1t+b_1)/(a_2t+b_2) = 0,\infty$.
We can make a linear change of variables so that this becomes $t_1t_2t_3t_4=1$, and the node is $\phi(0)=\phi(\infty)$.

Otherwise $f(t) = (a_1 t + b_1)(a_2 t + b_2)^3$.
Since the catalecticant is an invariant of the fundamental quartic of $\delta$, we can use a linear change of variables to assume $f(s) = (cs-d)s^3$, for some $c,d \in \CC$.
We then obtain the polarisation
\begin{equation*}
F(s_1, s_2, s_3, s_4) = cs_1s_2s_3s_4 - d \sum_{i < j <k} s_is_js_k = 0,
\end{equation*} 
which can be written as $\phi(s_1) + \phi(s_2) + \phi(s_3) + \phi(s_4) = 0$, where
\begin{equation*}
\phi(s) = \frac{c}{4} - \frac{d}{s}.
\end{equation*}
This is the case where $\delta$ is cuspidal.
\end{proof}


\end{document}